\documentclass[a4paper]{article}

\usepackage{amsmath,amssymb,color}

\usepackage{graphicx}
\usepackage{subfigure}
\usepackage{subdepth}
\usepackage{stmaryrd} 
\usepackage{lineno,etoolbox,units}

\allowdisplaybreaks
\DeclareGraphicsExtensions{.pdf,.eps}

\newcommand{\vertiii}[1]{{\left\vert\kern-0.25ex\left\vert\kern-0.25ex\left\vert #1 
    \right\vert\kern-0.25ex\right\vert\kern-0.25ex\right\vert}}
\newcommand{\bu}{\boldsymbol u}

\newcommand{\bv}{\boldsymbol v}
\newcommand{\bV}{\boldsymbol V}
\newcommand{\bD}{\boldsymbol D}
\newcommand{\bW}{\boldsymbol W}
\newcommand{\bZ}{\boldsymbol Z}
\newcommand{\bw}{\boldsymbol w}
\newcommand{\bbeta}{\boldsymbol  \eta}

\newcommand{\bz}{\boldsymbol z}

\newcommand{\bn}{\boldsymbol n}
\newcommand{\bm}{\boldsymbol m}

\newcommand{\balpha}{\boldsymbol \alpha}

\newcommand{\bpsi}{\boldsymbol \psi}

\newcommand{\be}{\boldsymbol e}

\newcommand{\bvar}{\boldsymbol \varphi}

\newcommand{\bs}{\boldsymbol s}

\newcommand{\bff}{\boldsymbol f}

\newtheorem{Theorem}{Theorem}[section]
\newtheorem{lema}[Theorem]{Lemma}
\newtheorem{proposition}[Theorem]{Proposition}
\newtheorem{corollary}[Theorem]{Corollary}

\newtheorem{remark}[Theorem]{Remark}

\newtheorem{Proof}{{\em Proof:}}

\newenvironment{proof}{\begin{Proof}\rm}{\hfill $\Box$ \end{Proof}}

\newcommand{\linenomathpatch}[1]{%
  \cspreto{#1}{\linenomath}%
  \cspreto{#1*}{\linenomath}%
  \csappto{end#1}{\endlinenomath}%
  \csappto{end#1*}{\endlinenomath}%
}
\linenomathpatch{equation}
\linenomathpatch{gather}
\linenomathpatch{multline}
\linenomathpatch{align}
\linenomathpatch{alignat}
\linenomathpatch{flalign}

\usepackage{hyperref}

\usepackage{macros}

\title{Pressure and convection robust bounds for continuous interior penalty  divergence-free finite element methods
for the incompressible Navier-Stokes equations}
\author{Bosco
Garc\'{\i}a-Archilla\thanks{Departamento de Matem\'atica Aplicada
II, Universidad de Sevilla, Sevilla, Spain. Research is supported by
Spanish MCINYU under grants PGC2018-096265-B-I00 and PID2019-104141GB-I00 (bosco@esi.us.es)}
  \and Julia Novo\thanks{Departamento de
Matem\'aticas, Universidad Aut\'onoma de Madrid, Spain. Research is supported
by Spanish MINECO
under grants PID2019-104141GB-I00 and VA169P20  (julia.novo@uam.es)}}
\date{\today}
\begin{document}

\maketitle

\begin{abstract}
 In this paper we analyze a pressure-robust method based on divergence-free mixed finite element methods with
 continuous interior penalty stabilization. The main goal is to prove an $O(h^{k+1/2})$ error estimate for
 the $L^2$ norm of the velocity in the convection dominated regime. This bound is pressure robust (the error
 bound of the velocity does not depend on the pressure) and also convection robust (the constants in the 
 error bounds are independent of the Reynolds number). 
 \end{abstract}


\section{Introduction}
We consider finite element approximations to the incompressible Navier-Stokes equations. Our aim is to propose a method that is able to get both pressure-robust error bounds and convection robust error bounds. Pressure-robust error bounds are those for which the error bounds of the velocity are independent on Sobolev norms of the true pressure. Convection robust error bounds are those in which the constants in the error bounds are independent of the Reynolds number (see e.g., \cite{review} for the practical implications of this property).
For the first property, we need to use divergence free pairs
of finite elements spaces. Using these pairs there is an exact conservation of mass and the numerical solution satisfies a so-called fundamental invariance property: gradient forces on the right-hand side possess an impact only on the discrete pressure. We refer to 
\cite{john_et_al_rev} for a comprehensive discussion of these kind of methods.
In the recent review \cite{review}, there is a deep study about finite element methods for the Navier-Stokes equations having the convection robust property. In \cite{review} those methods are classified in three sets depending on the rate of convergence achieved for the $L^2$ norm of the velocity: $O(h^{k-1})$, $O(h^k)$ and $O(h^{k+1/2})$ for methods based on piecewise polynomials for the velocity of degree $k$. As explained in \cite{review}, the exact conservation of mass has an influence also in the convection robustness of the bounds. More precisely, the plain Galerkin method, without any kind of stabilization, with divergence-free finite elements possesses pressure-robust and convection robust bounds of order $O(h^k)$, whereas no such bounds seem to be possible for non divergence-free elements. The question about the possibility of getting
order $O(h^{k+1/2})$ for these divergence-free methods adding some kind of stabilization has been open for a while. 

Recently, the authors of \cite{barren_cip} propose a continuous interior penalty stabilization for divergence-free methods that is able to give a method
of order $k+1/2$ and pressure robust bounds when applied to a steady Oseen equation. The stabilization in \cite{barren_cip} is inspired in the stabilization proposed in \cite{ahmed_etal}. The authors in \cite{ahmed_etal} proposed and analyzed a
method for the steady Oseen equations for which pressure and convection robust bounds of order $k+1/2$ are proved. For the
method in \cite{ahmed_etal} the authors use divergence-free elements and add 
a bulk term in the form of a residual-based least squares stabilization of the vorticity equation plus a continuous interior penalty stabilization term. The main idea in \cite{barren_cip} is to replace the bulk term by two more continuous penalty stabilization terms based on higher order derivatives.

In the present paper, we follow the ideas of \cite{barren_cip} to extend the method to the incompressible Navier-Stokes equations. To this end, the convective field in the stabilization terms in \cite{barren_cip} is
replaced by the velocity approximation. To prove the rate of convergence $k+1/2$ for $k\ge 4$ we
need some a priori bounds for the velocity approximation and the velocity error. For this reason, in the error analysis of
the present paper that works for any value of $k$, we first prove a rate of convergence of order $k$ for the method without any a priori bound and then used the previous proved bounds to achieved the optimal rate of convergence $k+1/2$. 

While writing this paper we found a related work, see \cite{da_veiga_etal}, in which the authors obtain pressure and convection robust bounds of order $k+1/2$ for the incompressible Navier-Stokes equations in 2 space dimensions. In \cite{da_veiga_etal} a divergence-free pair of mixed finite elements is applied with a discontinuous Galerkin in time approximation of order one. The authors
in \cite{da_veiga_etal} add a SUPG-curl stabilization that is also inspired in the method proposed in \cite{ahmed_etal} for
the steady Oseen problem. As the authors state, they need a discontinuous method in time to be able to prove error bounds with the SUPG-curl residual method. There are again different stabilization terms in the method proposed in \cite{da_veiga_etal}, one of them involves the jumps
of the gradients along the faces, and is then analogous to the first continuous interior penalty stabilization term in \cite{ahmed_etal}, \cite{barren_cip} and the present paper. The other stabilization is a residual SUPG-curl stabilization. Comparing with \cite{da_veiga_etal}, we propose a simpler method that can be accompanied by any time integrator. Also, as stated in 
\cite{da_veiga_etal}, the error analysis of a continuous interior penalty method reflecting pressure robustness for the incompressible Navier-Stokes equations is a major challenge that did not exist in the literature. This is exactly
the question we try to solve with the error analysis of the present paper. 

\section{Preliminaries and notation}
Along the paper we follow the notation in \cite{barren_cip}.

This paper studies incompressible flow problems that are modeled by the
incompressible Navier--Stokes equations
\begin{equation}
\begin{array}{rcll}
\label{NS} \partial_t\bu -\nu \Delta \bu + (\bu\cdot\nabla)\bu + \nabla p &=& \bff &\text{in }\ (0,T]\times\Omega, \\
\nabla \cdot \bu &=&0&\text{in }\ (0,T]\times\Omega,
\end{array}
\end{equation}
in a bounded domain $\Omega \subset {\mathbb R}^d$, $d \in \{2,3\}$.
The boundary of $\Omega$ is assumed to be polyhedral and Lipschitz.
In~\eqref{NS},
$\bu$ denotes the velocity field, $p$ the kinematic pressure, $\nu>0$ the kinematic viscosity coefficient,
 and $\bff$ represents the accelerations due to external body forces acting
on the fluid. The Navier--Stokes equations \eqref{NS} have to be complemented
with an initial condition $\bu(0)=\bu^0$ and with
boundary conditions. For simplicity,
we only consider the case of homogeneous
Dirichlet boundary conditions $\bu = \boldsymbol 0$ on $[0,T]\times \partial \Omega$.

We define the space
$$
{\boldsymbol{\cal V}}=\left\{\bv\in H_0^1(\Omega)^d: {\rm div} \bv=0\right\}.
$$
It is known that a potential can be associated to every divergence-free function in $\Omega$. The space that captures the
kernel of the divergence operator is \cite{costabel}
$$
{\bZ}=\left\{\bz: \ {\rm components}\ {\rm of}\ \bz \ {\rm in}\ H^1(\Omega), {\rm curl}\bz\in H_0^1(\Omega)^d\right\}.
$$
If $\bu\in H^r(\Omega)^d\cap{\boldsymbol{\cal V}}$ then there exists $\bz\in \bZ$ with components in $H^{r+1}(\Omega)$ such that
$$
{\rm curl \bz}=\bu,\quad \|\bz\|_{H^{r+1}(\Omega)}\le C \|\bu\|_{H^r(\Omega)^d}.
$$
Although $\bz$ is a scalar for $d=2$ we keep the bold notation for both dimensions. Observe that for $d=3$ the above norm
of $\bz$ should be $\|\bz\|_{H^{r+1}(\Omega)^d}$. For simplicity, in the sequel we will denote this norm both for scalar
or vectorial functions as:
$$
\|\bz\|_{k+1}:=\|\bz\|_{H^{r+1}(\Omega)^d},
$$
and we will use the notation
$$
\|\bz\|_{k+1,K}:=\|\bz\|_{H^{r+1}(K)^d},\quad K\subset \Omega.
$$
Let us denote by  $\{\mathcal{T}_{h}\}=\{(K_j,\phi_{j}^{h})_{j \in J_{h}}\}$, $h>0$, a family of partitions of
$\overline\Omega$, where $h$ is the maximum diameter of the mesh cells $K_j\in \mathcal{T}_{h}$
and $\phi_j^h$ are the mappings from the reference simplex $K_0$ onto $K_j$.
We assume that the family of partitions is shape-regular and  quasi-uniform. For $K\in \mathcal{T}_{h}$ we denote
by ${\cal F}_K$ the set of its faces. The collection of faces from the triangulation is denoted by $\cal F$ and
the collection of interior faces by ${\cal F}^i$. For
$F\in {\cal F}$ we denote by $h_F={\rm diam}(F)$ and by $|F|$ the measure of $F$. For a vector valued function $\bv$
we define the tangetial jumps across $F=K_1\cap K_2$, $K_1,K_2\in \mathcal{T}_{h}$ as
$$
\llbracket \bv\times \bn \rrbracket:=\bv_1\times \bn_1+\bv_2\times\bn_2
$$
where $\bv_i=\bv\mid_{K_i}$ and $\bn_i$ is the unit normal pointing outwards $K_i$. If $F$ is a boundary face we define
$$
\llbracket \bv\times \bn \rrbracket:=\bv\times \bn.
$$
We define the following broken inner products
\begin{eqnarray}\label{broken_in}
(\bv,\bw)_h:=\sum_{K\in \mathcal{T}_{h}}(\bv,\bw)_K,\quad <\bv,\bw>_{\cal F}:=\sum_{F\in {\cal F}}<\bv,\bw>_F,
\end{eqnarray}
with associated norms $\|\cdot\|_h,\|\cdot\|_{\cal F}$. Let us observe that the second norm can also be defined over
${\cal F}^i$.
We also consider the semi-norms
$$
|\bu|_{k,h}:=\sum_{|\alpha|=k}\|D^\alpha \bu\|_h.
$$

On $\mathcal T_h$, the following finite element spaces are defined. For $l\ge 1$ the standard piecewise continuous space
\begin{eqnarray*}
\bW_h^l=\left\{\bw_h\in H_0^1(\Omega)^d\ \mid \ {\bw_h}_{\mid_K}\in {\Bbb P}_l(K),\ \forall\ K\in \mathcal T_h\right\},
\end{eqnarray*}
and for $l\ge 0$ the discontinuous spaces
\begin{eqnarray*}
\bD_h^l&=&\left\{\bw_h\in L^2(\Omega)^d\ \mid \ {\bw_h}_{\mid_K}\in {\Bbb P}_l(K)^3,\ \forall\ K\in \mathcal T_h\right\}\\
D_h^l&=&\left\{\bw_h\in L^2(\Omega)\ \mid \ {\bw_h}_{\mid_K}\in {\Bbb P}_l(K),\ \forall\ K\in \mathcal T_h\right\}.
\end{eqnarray*}
For $k\ge 2$ we assume there exists finite element spaces $\bV_h\subset H_0^1(\Omega)^d$, $Q_h\subset L_0^2(\Omega)$ and
$$
{\boldsymbol{\cal V}}_h=\left\{\bv_h\in \bV_h\mid {\rm div}(\bv_h)=0\quad {\rm in}\ \Omega\right\}.
$$ 

The inverse inequality
\begin{equation}
\label{inv} | v_{h} |_{W^{m,p}(K)} \leq c_{\rm{inv}}
h_K^{n-m-d\left(\frac{1}{q}-\frac{1}{p}\right)}\quad 
| v_{h}|_{W^{n,q}(K)} \quad \forall\ v_{h} \in D_{h}^{l},
\end{equation}
holds, 
with $0\leq n \leq m \leq 1$, $1\leq q \leq p \leq \infty$, and $h_K$
being the diameter of~$K \in \mathcal T_h$, because the family of partitions is quasi-uniform, e.g., see \cite[Theorem~3.2.6]{Cia78}.

As in \cite{barren_cip} we assume
\begin{enumerate}
\item ${\rm div}(\bV_h)\subset Q_h$.
\item  The pair $(\bV_h,Q_h)$ is inf-sup stable.
\item  $\bW_h^k\subset \bV_h\subset \bW_h^r$, for some $k,r\ge 2$.
\item  There exists a finite element space $\bZ_h\subset\bZ$ such that ${\rm curl}(\bZ_h)={\boldsymbol{\cal V}}_h$.
\item For $P_h$ the orthogonal projection onto $\bZ_h$ the following error estimate holds for $\bz\in \bZ$
with components in $H^{k+2}(\Omega)$ and $|\alpha|\le k+1$
\begin{equation}\label{cota_Ph_z}
\|h^{|\alpha|}\partial^\alpha (\bz-P_h\bz)\|_h\le C h^{k+2}\|\bz\|_{k+2}.
\end{equation}
\end{enumerate}
As shown in \cite{barren_cip} the Scott-Vogelious finite element spaces on
barycenter-refined  triangulations verify the above
assumptions with velocity and pressure spaces, $k\ge 2$
\begin{eqnarray*}
\bV_h=\bW_h^k\quad Q_h=D_h^{k-1}\cap L_0^2(\Omega).
\end{eqnarray*}
The potential space is given by
$$
\bZ_h^{k+1}=\left\{\bz_h\in C^1(\Omega),\ \bz_h\mid_K\in {\Bbb P}_{k+1}(K),\ K\in \mathcal{ T}_{h}\right\}.
$$
In the rest of the paper we will consider these finite element spaces.
For these spaces apart from the above assumptions the following Lemma can be found in \cite[Lemma 4.1]{barren_cip}.
\begin{lema}\label{lema_41} Let $m\ge 3$. There exists $C>0$, independent of $h$, such that for all 
$\bw_h\in \bD_h^{m}$ it holds
$$
\inf_{\bm_h\in \bZ_h^m}\|h^{3/2}(\bw_h-\bm_h)\|_h\le C_{\ref{lema_41}}\left(||h^2\llbracket \bw_h\rrbracket||_{{\cal F}^i}
+||h^3\llbracket \nabla\bw_h\rrbracket||_{{\cal F}^i}\right).
$$
\end{lema}
We remark that the above result is proved in~\cite{barren_cip} only for $d=2$, although the authors indicate that the proof can be extended to the case~$d=3$ arguing as in~\cite{21}.

 Let  $\bs_h$ be the modified Stokes projection defined by
$$
 \displaylines{
 (\nabla \bs_h,\nabla \bvar_h)  = (\nabla \bu,\nabla \bvar_h),\qquad \bvar_h \in {\boldsymbol{\cal V}}_h,
 }
 $$
 for which the following bound holds \cite[(22)]{nos_oseen}
 \begin{equation}
 \label{err_sh}
 \left\| \bs_h-\bu\right\|_0 + h\left\| \nabla(\bs_h -\bu)\right\|_0 \le C_{\bs}\left\|\bu\right\|_{k+1}h^{k+1}.
 \end{equation}
Following \cite{Chen} one can also prove
\begin{equation}\label{sto_inf}
\|\bs_h\|_\infty\le C_{\bs} \|\bu\|_\infty,\quad \|\nabla \bs_h\|_\infty\le C_{\bs} \|\nabla \bu\|_\infty.
\end{equation}
For the velocity $\bu$ in \eqref{NS} we denote by $\bz$ its corresponding potential. Then,
we define $\bpsi_h=P_h \bz$ and $\bw_h={\rm curl} \bpsi_h$.
Adding and subtracting $\bs_h$ and applying \eqref{sto_inf}, \eqref{inv}, \eqref{cota_Ph_z} and \eqref{err_sh} it follows that
  \begin{equation}
 \label{nabla_w_h_inf} \left\| \nabla\bw_h\right\|_\infty \le  C_{\bs} \left\|\nabla \bu\right\|_\infty + Ch^{k-d/2}\left\|\bz\right\|_{k+2}.
 \end{equation}
 Arguing similary,
  \begin{equation}
 \label{w_h_inf} \left\| \bw_h\right\|_\infty \le C_{\bs} \left\| \bu\right\|_\infty + Ch^{k+1-d/2}\left\|\bz\right\|_{k+2}.
 \end{equation}
 Thus, in the sequel we denote
 \begin{equation}
 \label{C1wh}
 C_{1,\bw_h}=\max_{0\le t\le T} \left\|\nabla\bw_h(t)\right\|_{\infty},
 \end{equation}
  \begin{equation}
 \label{Cwh}
 C_{\bw_h}=\max_{0\le t\le T} \left\|\bw_h(t)\right\|_{\infty}.
 \end{equation}
\section{Error analysis}
The stabilized finite element approximation is defined as follows. 

Find: $(\bu_h(t),p_h(t))\in (\bV_h,Q_h)$, $t\in (0,T]$ such that
for all $(\bv_h,q_h)\in (\bV_h,Q_h)$ it holds
\begin{eqnarray}\label{metodo}
(\bu_{h,t},\bv_h)+\nu(\nabla \bu_h,\nabla \bv_h)+(\bu_h\cdot\nabla \bu_h,\bv_h)+S(\bu_h,\bv_h)
-(p_h,\nabla \cdot\bv_h)\nonumber\\=(\bff,\bv_h),\\
(\nabla \cdot \bu_h,q_h)=0,\nonumber
\end{eqnarray}
where, following \cite{barren_cip} the stabilized term is defined as
$$
S(\bu_h,\bv_h)=\|\bu_h\|_{0,\infty}^{-1}\sum_{j=1}^3 \delta_j S_j(\bu_h,\bv_h),
$$
where
\begin{eqnarray}\label{s1}
S_1(\bu_h,\bv_h)&=&<h^2 \llbracket (\bu_h\cdot \nabla )\bu_h\times \bn\rrbracket,\llbracket (\bu_h\cdot \nabla )\bv_h\times \bn\rrbracket>_{{\cal F}^i}\\
\label{s2}
S_2(\bu_h,\bv_h)&=&<h^4 \llbracket{\cal B}\bu_h\rrbracket,\llbracket{\cal B}\bv_h\rrbracket>_{{\cal F}^i}\\
\label{s3}
S_3(\bu_h,\bv_h)&=&<h^6 \llbracket\nabla{\cal B}\bu_h\rrbracket,\llbracket\nabla{\cal B}\bv_h\rrbracket>_{{\cal F}^i},
\end{eqnarray}
where 
$$
({\cal B}\bw)\mid_K={\rm curl}((\bu_h\cdot\nabla)\bw)\mid_K,
$$
and $\delta_j$, $j=1,2,3$ are non-dimensional stabilization parameters.

We introduce two constants, a characteristic length $l_0$ (which, for the purpose of simplicity, we may choose as the diameter of~$\Omega$ or $\left|\Omega\right|^{1/d}$) and a characteristic time $\tau_0$ (which for the purpose of simplicity we may choose as $\tau_0=T$) with which we define the constant
\begin{equation}
v_0=\frac{l_0}{\tau_0}.
\label{v_0}
\end{equation}
We define the following norm for a regular enough function
\begin{equation}\label{estrella}
\|\bz\|_*^2=(\nu+v_0h)\sum_{j=0}^{\min(5,k+2)}h^{2j-4}\|D^j\bz\|_h^2.
\end{equation}
We recall the following local trace inequality: there exist a positive constant~$C_{\rm tr}$ such that for any~$K\in\mathcal{T}_{h}$ and $F\in {\cal F}_K$ and any~$v\in H^1(K)$, it holds
\begin{equation}
\label{trace}
\left\| v\right\|_{0,F} \le C_{\rm tr}\left(h_K^{-1/2} \left\| v\right\|_{0,K} + h_K^{1/2} \left|v\right|_{1,K}\right).
\end{equation}
We denote by
$$
\delta_{\min} = \min\{\delta_1,\delta_2,\delta_3\},\qquad
\delta_{\max} = \max\{\delta_1,\delta_2,\delta_3\}.
$$
Following \cite{barren_cip}, for the velocity $\bu$ in \eqref{NS}, we denote by $\bz$ its corresponding potential. Then,
we define $\bpsi_h=P_h \bz$ and $\bw_h={\rm curl} \bpsi_h$. Finally we denote by
$$
\be_h=\bu_h-\bw_h,\quad \bbeta_h=\bu-\bw_h. 
$$
We first observe that using the above definitions and \eqref{NS} we get for any $\bv_h\in {\boldsymbol{\cal V}}_h$
\begin{eqnarray}\label{wh}
(\bw_{h,t},\bv_h)+\nu(\nabla \bw_h,\nabla\bv_h)+(\bu\cdot\nabla \bu,\bv_h)+S(\bw_h,\bv_h)\\
\quad=S(\bw_h,\bv_h)-(\bbeta_t,\bv_h)-\nu(\nabla \bbeta,\nabla \bv_h)+(\bff,\bv_h).\nonumber
\end{eqnarray}
Subtracting \eqref{wh} from \eqref{metodo} and taking $\bv_h=\be_h$ we get
\begin{eqnarray}\label{eq_error}
&&\frac{1}{2}\frac{d}{dt}\|\be_h\|_0^2+\nu\|\nabla \be_h\|_0^2+S(\be_h,\be_h)=(\bu\cdot\nabla \bu-\bu_h\cdot \nabla \bu_h,\be_h)+S(\bw_h,\be_h)\nonumber\\
&&\quad +(\bbeta_{h,t},\be_h)+\nu(\nabla \bbeta_h,\nabla \be_h).
\end{eqnarray}
For the nonlinear term we write
\begin{eqnarray*}
(\bu\cdot\nabla \bu-\bu_h\cdot \nabla \bu_h,\be_h)=((\bu-\bu_h)\cdot \nabla \bu,\be_h)+(\bu_h\cdot\nabla(\bu-\bu_h),\be_h).
\end{eqnarray*}
And adding $(\bu_h\cdot\nabla \be_h,\be_h)=0$, we may write
\begin{eqnarray*}
(\bu\cdot\nabla \bu-\bu_h\cdot \nabla \bu_h,\be_h)=((\bu-\bu_h)\cdot \nabla \bu,\be_h)+(\bu_h\cdot\nabla\bbeta_h,\be_h)=
\nonumber\\
(\bbeta_h\cdot\nabla \bu,\be_h)-(\be_h\cdot\nabla \bu,\be_h)+(\bu_h\cdot\nabla\bbeta_h,\be_h).
\end{eqnarray*}
Inserting the above equality into \eqref{eq_error} we finally obtain
\begin{eqnarray}\label{eq_error_2}
&&\frac{1}{2}\frac{d}{dt}\|\be_h\|_0^2+\nu\|\nabla \be_h\|_0^2+S(\be_h,\be_h)=(\bbeta_h\cdot\nabla \bu,\be_h)-(\be_h\cdot\nabla \bu,\be_h)\nonumber\\
&&\quad+(\bu_h\cdot\nabla\bbeta_h,\be_h)+S(\bw_h,\be_h) +(\bbeta_{h,t},\be_h)+\nu(\nabla \bbeta_h,\nabla \be_h).
\end{eqnarray}
Now, for the first and second terms on the right-hand side of \eqref{eq_error_2} we get
\begin{eqnarray*}
|(\bbeta_h\cdot\nabla \bu,\be_h)|+|(\be_h\cdot\nabla \bu,\be_h)|\le \frac{3}{2}\|\nabla \bu\|_\infty\|\be_h\|_0^2+\frac{1}{2}\|\nabla \bu\|_\infty\| \bbeta_h\|_0^2.
\end{eqnarray*}
And since $\bbeta_h={\rm curl}(\bz-\bpsi_h)$, assuming $\nabla \bu\in L^\infty(\Omega)^d$ and  using definition \eqref{estrella} we can write
\begin{eqnarray}\label{part1}
|(\bbeta_h\cdot\nabla \bu,\be_h)|+|(\be_h\cdot\nabla \bu,\be_h)|\le \frac{3}{2}\|\nabla \bu\|_\infty\|\be_h\|_0^2+\frac{h}{2v_0}\left\|\nabla\bu\right\|_\infty\|\bz-\bpsi_h\|_*^2.
\end{eqnarray}
For the last two terms on the right-hand side of \eqref{eq_error_2} we obtain
\begin{eqnarray*}
|(\bbeta_{h,t},\be_h)|+\nu|(\nabla \bbeta_h,\nabla \be_h)|\le \frac{T}{2}\|\bbeta_{h,t}\|_0^2+\frac{\nu}{2}\|\nabla\bbeta_h\|_0^2
+\frac{1}{2T}\|\be_h\|_0^2+\frac{\nu}{2}\|\nabla \be_h\|_0^2.
\end{eqnarray*}
Using again definition \eqref{estrella} and applying \eqref{cota_Ph_z} to $\bz_t$ we get
\begin{eqnarray}\label{part2}
|(\bbeta_{h,t},\be_h)|+\nu|(\nabla \bbeta_h,\nabla \be_h)|&\le& C{T}h^{2k+2}\|\bu_t\|_{k+1}+C\|\bz-\bpsi_h\|_*^2\nonumber\\
&&\quad+\frac{1}{2T}\|\be_h\|_0^2+\frac{\nu}{2}\|\nabla \be_h\|_0^2.
\end{eqnarray}
Inserting \eqref{part1} and \eqref{part2} into \eqref{eq_error_2}  we reach
\begin{eqnarray}\label{eq_error_3}
\frac{d}{dt}\|\be_h\|_0^2+{\nu}\|\nabla \be_h\|_0^2+S(\be_h,\be_h)\le\left(3\|\nabla \bu\|_\infty+\frac{1}{T}\right){\|\be_h\|_0^2}+2(\bu_h\cdot\nabla\bbeta_h,\be_h)\nonumber\\
\quad+S(\bw_h,\bw_h)+CTh^{2k+2}\|\bu_t\|_{k+1}+C\left(\frac{h	}{v_0}\left\|\nabla\bu\right\|_\infty+1\right)
\|\bz-\bpsi_h\|_*^2.
\end{eqnarray}
To finish the error bound we need to bound the second and third terms on the right-hand side above. To this end we will apply some lemmas.

We will need estimations of $\bu_h-\bu$ and its first and second derivatives in different norms that we obtain next.
\begin{lema}\label{estimations} The following bounds hold for any multi-index $\balpha$ with $\left|\balpha\right| \le 2$,
\begin{align}\label{028}
\left\| h^{\left|\balpha\right|+1/2} \partial^{\balpha}(\bu_h-\bu) \right\|_{h}
 &\le
c_{\rm inv} h^{1/2} \left\|\be_h\right\|_0+ h^{\left|\balpha\right|+1/2} \left\|D^{\left|\balpha\right|+1}(\bz-\bpsi_h)\right\|_{h},
\\
\label{029}
\left\| h^{\left|\balpha\right|+3/2}\nabla  \partial^{\balpha}(\bu_h-\bu) \right\|_{h}
 &\le
c_{\rm inv} h^{1/2} \left\|\be_h\right\|_0+ h^{\left|\balpha\right|+3/2} \left\|D^{\left|\balpha\right|+2}(\bz-\bpsi_h)\right\|_{h}.
\end{align}
\end{lema}
\begin{proof}
 We prove~(\ref{028}--\ref{029}) for 
$(\bu_h-\bu)_{xy}$  since the rest of the second  and first order derivatives are dealt with similarly.
By writing
$$
(\bu-\bu_h)_{yx} = (\be_h)_{yx}+(\bbeta_h)_{yx},
$$
and applying the inverse inequality \eqref{inv} we have,
\begin{align}\label{028b}
\left\| h^{5/2} (\bu_h-\bu)_{yx} \right\|_{h}
 &\le
c_{\rm inv} h^{1/2} \left\|\be_h\right\|_0+h^{5/2} \left\|D^3(\bz-\bpsi_h)\right\|_{h},
\\
\label{029b}
\left\| h^{7/2} \nabla(\bu_h-\bu)_{yx} \right\|_{h} 
 &\le  c_{\rm inv} h^{1/2} \left\|\be_h\right\|_0
+h^{7/2} \left\|D^4(\bz-\bpsi_h)\right\|_{h}.
\end{align}
\end{proof}

\begin{lema}\label{cons2} (Consistency of the stabilization term). 
Assume that $\bu \in H^{\min(4,k+1)}(\Omega)$. Then, the following bound holds
\begin{align}
\label{eq:cons2} S(\bw_h,\bw_h) \le& C\frac{\delta_{\max}}{\left\|\bu_h\right\|_\infty} \left(
\left\| \nabla\bw_h\right\|_\infty^2 h\left\|\be_h\right\|_0^2 \vphantom{+ 
\left( \left\|\bu\right\|_\infty + h\left(\left\|\nabla\bu\right\|_\infty + C_{1,\bw_h}\right)\right)
\frac{1}{v_0^{1/2}} \left\| \bz-\bpsi_h\right\|_{*}}
\right.
\nonumber\\
&\left.{}+
\left( \left\|\bu\right\|_\infty + h\left(\left\|\nabla\bu\right\|_\infty + \left\| \nabla\bw_h\right\|_\infty\right)\right)^2
\frac{1}{v_0} \left\| \bz-\bpsi_h\right\|_{*}^2
\right),
\end{align}
 where the constant~$C$ depends on~$C_{\rm tr}$, $c_{\rm{inv}}$
 and the dimension~$d$.
\end{lema}
 \begin{proof}
For the purpose of the present proof, for $\bw=\bu$ or~$\bw=\bu_h -\bu$, we denote by
$$
{\cal B}_{\bw}\bv =\hbox{\rm curl}\left((\bw \cdot \nabla )\bv\right),
$$
and
\begin{eqnarray*}
S_{\bw} (\bv,\bv) &=& \frac{1}{\left\|\bu_h\right\|_\infty} \left(\delta_1\left\|h \llbracket (\bw\cdot\nabla)\bv\times \bn\rrbracket \right\|_{{\cal F}_i}^2 + 
\delta_2\left\|h^2  \llbracket {\cal B}_{\bw}\bv \rrbracket\right\|_{{\cal F}_i}^2
\right.\\
&&\quad\left.+\delta_3\left\|h^3  \llbracket \nabla {\cal B}_{\bw}\bv\rrbracket \right\|_{{\cal F}_i}^2\right).
\end{eqnarray*}
We write
\begin{align}
\label{aver00}
S(\bw_h,\bw_h) &= \left(S(\bw_h,\bw_h) - S_{\bu}(\bw_h,\bw_h)\right) + S_{\bu} (\bw_h,\bw_h)
\nonumber\\
&=S_{\be}(\bw_h,\bw_h) + S_{\bu}(\bw_h,\bw_h).
\end{align}
where
\begin{equation}
\label{be}
\be=\bu_h-\bu.
\end{equation}
Applying the trace inequality~\eqref{trace} we can write
\begin{align}
\label{aaver00}
\left\|h \llbracket (\be\cdot\nabla)\bw_h\times \bn\rrbracket \right\|_{{\cal F}_i}
&\le C_{tr}\left( \left\|h^{1/2} (\be\cdot\nabla)\bw_h\right\|_{h} + \left\|h^{3/2} \nabla (\be\cdot\nabla)\bw_h\right\|_{h}\right)
\nonumber
\\
&\le C_{tr}\left( \left\|h^{1/2} (\be\cdot\nabla)\bw_h\right\|_{h} + \left\|h^{3/2} \nabla \be\cdot\nabla\bw_h\right\|_{h}\right.
\nonumber\\
&\qquad\left.{}
+ \left\|h^{3/2} \be\cdot\hbox{\rm Hess}(\bw_h)\right\|_{h} \right),
\end{align}
where, by $\be\cdot\hbox{\rm Hess}(\bw_h)$ we denote a matrix whose columns are the inner product of~$\be$ with the Hessian of each of the components of~$\bw_h$.
For the last term on the right-hand side above we apply the inverse inequality \eqref{inv} to the second
derivatives of~$\bw_h$ to get
\begin{align*}
\left\|h^{3/2} \be\cdot\hbox{\rm Hess}(\bw_h)\right\|_{h} &\le h^{1/2}\left\|\be\right\|_0c_{\rm inv} \left\| \nabla\bw_h\right\|_\infty
\nonumber\\
&{} \le\left( c_{\rm inv}h^{1/2}\left\|\be_h\right\|_0 + c_{\rm inv}\frac{ h}{v_0^{1/2}} \left\| \bz -\psi_h\right\|_{*}\right)\left\| \nabla\bw_h\right\|_\infty,
\end{align*}
for the second term of the right-hand side of~\eqref{aaver00} we apply~\eqref{029} to $\nabla\be_h$ to get
\begin{align*}
\left\|h^{3/2} \nabla \be\cdot\nabla \bw_h\right\|_{h} &\le 
\left(c_{\rm inv}h^{1/2}\left\|\be_h\right\|_0+\frac{ h}{v_0^{1/2}} \left\| \bz -\bpsi_h\right\|_{*} \right)\left\| \nabla\bw_h\right\|_\infty.
\end{align*}
Arguing similary for the first term on the right-hand side of~\eqref{aaver00} we obtain
\begin{align*}
\label{aaver01}
&\left\|h \llbracket (\be\cdot\nabla)\bw_h\times \bn\rrbracket \right\|_{{\cal F}_i}
\nonumber\\
&\quad\le C_{\rm tr}\left\| \nabla\bw_h\right\|_\infty(2+2c_{\rm inv}) \left( h^{1/2} \left\|\be_h\right\|_0 +  \frac{h}{v_0^{1/2}} \left\|\bz-\bpsi_h\right\|_{*}\right).
\end{align*}
Applying the trace inequaltiy~\eqref{trace} {again} we may write
\begin{eqnarray*}
\label{aaver02a}
\left\|h^2  \llbracket {\cal B}_{\be}\bw_h \rrbracket\right\|_{{\cal F}_i} \le C_{\rm tr} (I_{11}+I_{12}),
\end{eqnarray*}
where
 $I_{11}$ and~$I_{12}$ can be expressed as a sum of terms of the form
\begin{eqnarray*}
\label{alpha-beta}
\kappa_{\boldsymbol\beta} h^{\left|\balpha\right|+\left|{\boldsymbol\beta}\right|+1/2}
\left\| 
\partial^{\balpha} \be\cdot \partial^{\boldsymbol\beta}\nabla w_h
\right\|_h
\end{eqnarray*}
where $w_h$ represents any of the components of~$\bw_h$,
$\left|\balpha\right|+\left|{\boldsymbol\beta}\right|=1$ in the case of~$I_{11}$ and~$\left|\balpha\right|+\left|{\boldsymbol\beta}\right|=2$
in the case of~$I_{12}$ and
\begin{equation}
\label{kappa}
\kappa_{\boldsymbol\beta}=\left\{\begin{array}{lcl}1 & \quad& \hbox{\rm if $\left|{\boldsymbol\beta}\right|+1\le k$},\\
0&&\hbox{\rm otherwise.}\end{array}\right.
\end{equation}
Notice that the presence of~$\kappa_{\boldsymbol\beta}$ reflects the fact that, being $\bw_h$  piecewise polynomial of degree $k$, their derivatives of order $k+1$ are 0.

Applying the inverse inequality,
$$
\left\| \partial^{\boldsymbol\beta}\nabla w_h\right\|_\infty \le c_{\rm inv} h^{-\left|{\boldsymbol\beta}\right|}\left\|\nabla\bw_h\right\|_\infty,
$$
so that
\begin{eqnarray}
\label{I11+I12}
I_{11}+I_{12} \le Cc_{\rm inv}\left\| \nabla\bw_h\right\|_\infty \sum_{\left|\balpha\right|=0}^2 h^{\left|\balpha\right|+1/2}\left\| \partial^{\balpha} \be\right\|_h.
\end{eqnarray}
Consequently, applying~\eqref{028} it follows that
\begin{equation}
\label{aaver02}
\left\|h^2  \llbracket  {\cal B}_{\bw}\bv \rrbracket\right\|_{{\cal F}_i} \le 
CC_{\rm tr} \left\| \nabla\bw_h\right\|_\infty
\left( h^{1/2}\left\|\be_h\right\|_0 + \frac{ h}{v_0^{1/2}} \left\| \bz -\bpsi_h\right\|_{*} \right).
\end{equation}
Observe that the constant~$C$ in~\eqref{I11+I12} depends on the number of terms in $I_{11}$ and~$I_{12}$, which, in turn, depends on the dimension. Consequently, this is also the case of the constant~$C$ in~\eqref{aaver02}, which depends also on~$c_{\rm inv}$.

With respect $\left\|h^3  \llbracket \nabla {\cal B}_{\bw}\bv \rrbracket\right\|_{{\cal F}_i}$, the argument is the same as
with~$\left\|h^2  \llbracket {\cal B}_{\bw}\bv \rrbracket\right\|_{{\cal F}_i}$, with the only difference that, now,
$\left|\balpha\right|+\left|{\boldsymbol\beta}\right|=2,3$. Thus, we conclude that
\begin{equation}
\label{aaver03}
S_{\be} (\bw_h,\bw_h) \le C_2 \left\| \nabla\bw_h\right\|_\infty^2\frac{\delta}{\left\|\bu_h\right\|_\infty} \left( h^{1/2}\left\|\be_h\right\|_0 + \frac{ h}{v_0^{1/2}} \left\| \bz -\bpsi_h\right\|_{*} \right)^2.
\end{equation}
To bound~$S_{\bu}(\bw_h,\bw_h)$ we estimate the different terms separately. Since we are assuming that $\bu\in H^3(\Omega)^d$, it follows that  $(\bu\cdot\nabla)\bu \in H^2(\Omega)$ and, hence, is continuous, so that
$$
\left\|h \llbracket (\bu\cdot\nabla)\bu\times \bn\rrbracket \right\|_{{\cal F}_i}=0,
$$
and, consecuently
\begin{align}
\label{baver00b}
\left\|h \llbracket (\bu\cdot\nabla)\bw_h\times \bn\rrbracket \right\|_{{\cal F}_i}&=
\left\|h \llbracket (\bu\cdot\nabla)\bbeta_h\times \bn\rrbracket \right\|_{{\cal F}_i}
\nonumber\\
&{}\le C\left\|\bu\right\|_\infty\left\|h \llbracket \nabla\bbeta_h\rrbracket \right\|_{{\cal F}_i}.
\end{align}
Applying the trace inequality~\eqref{trace} we have
\begin{align}
\label{baver00c}
\left\|h \llbracket \nabla\bbeta_h\rrbracket \right\|_{{\cal F}_i} &\le
CC_{\rm tr} \left(h^{1/2} \left\| D^2(\bz-\bpsi_h)\right\|_h + h^{3/2} \left\| D^3(\bz-\bpsi_h)\right\|_h\right)
\nonumber\\
&{}\le CC_{\rm tr}\frac{1}{v_0^{1/2}} \left\|\bz-\bpsi_h\right\|_{*},
\end{align}
so that from~\eqref{baver00b},  it follows that
\begin{align}
\label{baver00}
\left\|h \llbracket (\bu\cdot\nabla)\bw_h\times \bn\rrbracket \right\|_{{\cal F}_i}
&{}\le CC_{\rm tr}\left\|\bu\right\|_\infty\frac{1}{v_0^{1/2}} \left\|\bz-\bpsi_h\right\|_{*},
\end{align}
where $C$ depends on the number of elements in $(\bu\cdot\nabla)\bw_h\times \bn$, that is, on the dimension.

Also, since we are assuming that $\bu\in H^3(\Omega)^d$, we have that ${\cal B}_{\bu}\bu\in H^1(\Omega)^d$, which implies (see e.g., \cite[p.~37]{libro_dis})
$$
\left\|h^2  \llbracket {\cal B}_{\bu}\bu \rrbracket\right\|_{{\cal F}_i}=0,
$$
so that
\begin{align}
\label{baver01b}
\left\|h^2  \llbracket {\cal B}_{\bu}\bw_h \rrbracket\right\|_{{\cal F}_i} &=
\left\|h^2  \llbracket {\cal B}_{\bu}\bbeta_h \rrbracket\right\|_{{\cal F}_i}
\nonumber\\
&{} \le Ch\left\| \nabla \bu\right\|_\infty \left\|h \llbracket \nabla\bbeta_h\rrbracket \right\|_{{\cal F}_i}
 + \left\|\bu\right\|_\infty
\left\| h^2  \llbracket \nabla\hbox{\rm curl}(\bbeta_h)\rrbracket \right\|_{{\cal F}_i}.
\end{align}
We notice that $ \left\|h \llbracket \nabla\bbeta_h\rrbracket \right\|_{{\cal F}_i}$ is bounded in~\eqref{baver00c}. Applying the trace inequality~\eqref{trace} 
we have
\begin{align*}
\left\| h^2  \llbracket \nabla\hbox{\rm curl}(\bbeta_h)\rrbracket \right\|_{{\cal F}_i}
&\le 
CC_{\rm tr} \left(h^{3/2} \left\| D^3(\bz-\bpsi_h)\right\|_h + h^{5/2} \left\| D^4(\bz-\bpsi_h)\right\|_h\right)
\nonumber\\
&{}\le CC_{\rm tr}\frac{1}{v_0^{1/2}} \left\|\bz-\bpsi_h\right\|_{*}.
\end{align*}
Thus, from \eqref{baver01b} it follows that
\begin{align}
\label{baver01}
\left\|h^2  \llbracket {\cal B}_{\bu}\bw_h \rrbracket\right\|_{{\cal F}_i} &\le 
CC_{\rm tr} \left(\left\|\bu\right\|_\infty + h\left\|\nabla\bu\right\|_\infty\right)
\frac{1}{v_0^{1/2}} \left\|\bz-\bpsi_h\right\|_{*},
\end{align}
where, as before, $C$ depends on the dimension.

Finally, we have that $\left\|h^3  \llbracket \nabla{\cal B}_{\bu}\bw_h \rrbracket\right\|_{{\cal F}_i}$ can be bounded by a sum of terms of the form
\begin{equation}
\label{terms2}
\kappa_{\boldsymbol\beta} \left\|h^3  \llbracket \partial^{\balpha}\bu\cdot \nabla\partial^{\boldsymbol\beta}\bw_h \rrbracket\right\|_{{\cal F}_i},
\end{equation}
where~$\kappa_{\boldsymbol\beta}$ is defined in~\eqref{kappa}, and~$\left|\balpha\right|+\left|{\boldsymbol\beta}\right|=2$.
We now argue why, for the terms
in~\eqref{terms2}, we have
\begin{equation}
\label{terms3}
\kappa_{\boldsymbol\beta} \left\|h^3  \llbracket \partial^{\balpha}\bu\cdot \nabla\partial^{\boldsymbol\beta}\bw_h \rrbracket\right\|_{{\cal F}_i}=
\kappa_{\boldsymbol\beta} \left\|h^3  \llbracket \partial^{\balpha}\bu\cdot \nabla\partial^{\boldsymbol\beta}\bbeta_h \rrbracket\right\|_{{\cal F}_i}.
\end{equation}
If $k=2$, then $\kappa_{\boldsymbol\beta}=0$ when ${\boldsymbol \beta}=2$, so that in terms with $\kappa_{\boldsymbol\beta}=1$ only up to second order derivatives appear, multiplied by first order derivatives,  so that, if $\left|{\boldsymbol\beta}\right|\le 1$ we have $ \partial^{\balpha}\bu\cdot \nabla\partial^{\boldsymbol\beta}\bu\in H^1(\Omega)^d$; thus (see e.g., \cite[p.~37]{libro_dis}), \eqref{terms3} holds. If $k\ge 3$, then we are assuming that~$\bu\in H^4(\Omega)$,  and then~$ \partial^{\balpha}\bu\cdot \nabla\partial^{\boldsymbol\beta}\bw_h\in H^1(\Omega)$, so that, again, \eqref{terms3} holds.

Now, arguing as before we have for $\left|\balpha\right|\le 1$
\begin{eqnarray*}
\label{aaver03a}
\kappa_{\boldsymbol\beta} \left\|h^3  \llbracket \partial^{\balpha}\bu\cdot \nabla\partial^{\boldsymbol\beta}\bbeta_h \rrbracket\right\|_{{\cal F}_i}
\le CC_{\rm tr} \left(\left\|\bu\right\|_\infty + h\left\|\nabla\bu\right\|_\infty\right)
\frac{1}{v_0^{1/2}} \left\|\bz-\bpsi_h\right\|_{*}.
\end{eqnarray*}
We deal differently with the case $\left|\balpha\right|=2$. Here we write
\begin{equation}
\label{aaver03b}
\left\|h^3  \llbracket \partial^{\balpha}\bu\cdot \nabla \bbeta_h \rrbracket\right\|_{{\cal F}_i}\le 
\left\|h^3  \llbracket \partial^{\balpha}\bbeta_h\cdot \nabla \bbeta_h \rrbracket\right\|_{{\cal F}_i} +
\left\|h^3  \llbracket \partial^{\balpha}\bw_h\cdot \nabla \bbeta_h \rrbracket\right\|_{{\cal F}_i}
\end{equation}
For the first term on the right-hand side of \eqref{aaver03b} since $\nabla \bw_h$ is a piecewise polynomial and $\nabla\bu$ is continuous,
we have that $\bbeta_h$ has only jump discontinuities, so that we can write 
\begin{align*}
\left\|h^3  \llbracket \partial^{\balpha}\bbeta_h\cdot \nabla \bbeta_h \rrbracket\right\|_{{\cal F}_i}&\le
h\left\| \nabla \bbeta_h\right\|_\infty\left\|h^2  \llbracket \partial^{\balpha}\bbeta_h \rrbracket\right\|_{{\cal F}_i}
\nonumber\\
&{}\le
h\left(\left\| \nabla \bu\right\|_\infty+\left\| \nabla\bw_h\right\|_\infty\right)\left\|h^2  \llbracket \partial^{\balpha}\bbeta_h \rrbracket\right\|_{{\cal F}_i},
\end{align*}
so that applying the trace inequality~\eqref{trace}
we get
\begin{align}
\label{aaver03c}
\left\|h^3  \llbracket \partial^{\balpha}\bbeta_h\cdot \nabla \bbeta_h \rrbracket\right\|_{{\cal F}_i}&\le CC_{\rm tr}
h\left(\left\| \nabla \bu\right\|_\infty+\left\| \nabla\bw_h\right\|_\infty\right)\left(h^{3/2} \left\|D^{2}\bbeta_h\right\|_h
+h^{5/2} \left\|D^{3}\bbeta_h\right\|_h\right)
\nonumber\\
&{}\le  CC_{\rm tr}
h\left(\left\| \nabla \bu\right\|_\infty+\left\| \nabla\bw_h\right\|_\infty\right) \frac{1}{v_0}^{1/2} \left\| \bz-\bpsi_h\right\|_{*}.
\end{align}
For the second term on the righ-hand side of~\eqref{aaver03b} we write
\begin{align*}
\left\|h^3  \llbracket \partial^{\balpha}\bw_h\cdot \nabla \bbeta_h \rrbracket\right\|_{{\cal F}_i}
&\le Ch^2\left\| \partial^{\balpha}\bw_h\right\|_\infty \left\|h\nabla \bbeta_h \rrbracket\right\|_{{\cal F}_i}
\nonumber\\
&\le Cc_{\rm inv}h\left\| \nabla \bw_h\right\|_\infty\left\|h\nabla \bbeta_h \rrbracket\right\|_{{\cal F}_i}.
\end{align*}
Thus applying again the trace inequality~\eqref{trace}
\begin{align}
\label{aaver03d}
\left\|h^3  \llbracket \partial^{\balpha}\bw_h\cdot \nabla \bbeta_h \rrbracket\right\|_{{\cal F}_i}
&\le Cc_{\rm inv}h\left\| \nabla\bw_h\right\|_\infty C_{\rm tr} 
\left(h^{1/2} \left\|D^{1}\bbeta_h\right\|_h
+h^{3/2} \left\|D^{2}\bbeta_h\right\|_h\right)
\nonumber\\
&{}\le Cc_{\rm inv}h\left\| \nabla\bw_h\right\|_\infty C_{\rm tr} \frac{1}{v_0^{1/2}}  \left\| \bz-\bpsi_h\right\|_{*}.
\end{align}
This, together with~\eqref{aaver03b}, implies that
\begin{eqnarray}
\label{aaver04}
\left\|h^3  \llbracket \nabla{\cal B}_{\bu}\bw_h \rrbracket\right\|_{{\cal F}_i}
\le 
CC_{\rm tr}\left( \left\|\bu\right\|_\infty + h\left(\left\|\nabla\bu\right\|_\infty + \left\| \nabla\bw_h\right\|_\infty\right)\right)
\frac{1}{v_0^{1/2}} \left\| \bz-\bpsi_h\right\|_{*},
\end{eqnarray}
where, as before, $C$ depends on the dimension and~$c_{\rm inv}$.
Consequently, from~\eqref{baver00}, \eqref{baver01}, and \eqref{aaver04} it follows that
\begin{equation}
\label{Suwhwh}
S_{\bu}(w_h,w_h) \le C\frac{\delta_{\max}}{\left\|\bu_h\right\|_\infty}\left( \left\|\bu\right\|_\infty + h\left(\left\|\nabla\bu\right\|_\infty + \left\| \nabla\bw_h\right\|_\infty\right)\right)^2
\frac{1}{v_0} \left\| \bz-\bpsi_h\right\|_{*}^2,
\end{equation}
which, together with \eqref{aver00} and~\eqref{aaver03}
finishes the proof.
 \end{proof}
As we stated in the introduction, it is well known that the standard Galerkin method based on Scott-Vogelious elements 
has a rate of convergence of order $k$ with constants that do not depend neither on the Reynolds number neither on the pressure. Thanks to the previous consistency lemma for the added stabilization term the same result holds for the stabilized method.
In the next proposition, as a first step to achieve the full order convergence of rate $k+1/2$ we prove that the stabilized
Scott-Vogelious method has at least rate of convergence $k$. This result will be used to obtained some a priori bounds
for the error that will be used in the proof of the main theorem.
\begin{proposition}\label{Prop0} Assume that
\begin{equation}
\label{u0}
\left\| \bu_h(t)\right\|_\infty\ge u_0>0,\qquad 0\le t\le T,
\end{equation}
and that
\begin{equation}
\label{cond_e0a}
\left\| \be_h(0)\right\|_0\le c_0h^k,
\end{equation}
for some constant~$c_0.$
Then, there exist a constant~$C_0$ such that the following inequality holds for $t\in [0,T]$
\begin{align}
\left\| \be_h(t) \right\|_0^2 
\le
C_0h^{2k}
\end{align}
\end{proposition}
\begin{proof}
Starting with the error equation \eqref{eq_error}
\begin{eqnarray*}
&&\frac{1}{2}\frac{d}{dt}\|\be_h\|_0^2+\nu\|\nabla \be_h\|_0^2+S(\be_h,\be_h)=(\bu\cdot\nabla \bu-\bu_h\cdot \nabla \bu_h,\be_h)+S(\bw_h,\be_h)\nonumber\\
&&\quad +(\bbeta_{h,t},\be_h)+\nu(\nabla \bbeta_h,\nabla \be_h),
\end{eqnarray*}
we bound the nonlinear term in a differente way as before. We first observe that using the skew-symmetric property of the bilinear term we get
\begin{eqnarray*}
(\bu\cdot\nabla \bu-\bu_h\cdot \nabla \bu_h,\be_h)=(\bbeta_h\cdot \nabla \bu,\be_h)+(\bw_h\cdot\nabla \bbeta_h,\be_h)-(\be_h\cdot\nabla \bw_h,\be_h)
\end{eqnarray*}
And then
\begin{eqnarray*}
|(\bu\cdot\nabla \bu-\bu_h\cdot \nabla \bu_h,\be_h)|&\le&\|\bbeta_h\|_0\|\nabla \bu\|_\infty\|\be_h\|_0+
\|\bw_h\|_\infty\|\nabla \bbeta_h\|_0\|\be_h\|_0\\
&&\quad+\|\be_h\|_0^2\|\nabla \bw_h\|_\infty\\
&\le&\left(\frac{\|\nabla \bu\|_\infty}{2}+\|\nabla \bw_h\|_\infty+\frac{1}{2T}\right)\|\be_h\|_0^2\\
&&\quad+\frac{\|\nabla\bu\|_\infty}{2}\|\bbeta_h\|_0^2+\frac{T\|\bw_h\|_\infty^2}{2} \|\nabla\bbeta_h\|_0^2.
\end{eqnarray*}
Arguing as in \eqref{eq_error_3} we get
\begin{eqnarray*}
&&\frac{d}{dt}\|\be_h\|_0^2+{\nu}\|\nabla \be_h\|_0^2+S(\be_h,\be_h)\le\left(\|\nabla \bu\|_\infty+2\|\nabla \bw_h\|_\infty+\frac{1}{T}\right){\|\be_h\|_0^2}\nonumber\\
&&\quad+S(\bw_h,\bw_h)+CTh^{2k+2}\|\bu_t\|_{k+1}+C\left(\frac{h	}{v_0}\left\|\nabla\bu\right\|_\infty+1\right)\|\bz-\bpsi_h\|_*^2
\\
&&\quad+C T\|\bw_h\|_\infty^2\|\nabla \bbeta_h\|_0^2.
\end{eqnarray*}
Applying now Lemma \ref{cons2} we can bound
\begin{eqnarray}\label{eq:Sww}
S(\bw_h,\bw_h)\le \left(C\frac{\delta_{\rm max}}{u_0}\|\nabla \bw_h\|_\infty^2 h\right)\|\be_h\|_0^2
+\frac{C_{\bu,\bw_h}}{v_0}\left\| \bz-\bpsi_h\right\|_{*}^2,
\end{eqnarray}
where
\begin{equation}\label{Cuw}
C_{\bu,\bw_h}=\left( \left\|\bu\right\|_{L^\infty(0,T,L^\infty)} + h\left(\left\|\nabla\bu\right\|_{L^\infty(0,T,L^\infty)} + C_{1,\bw_h}\right)\right)^2.
\end{equation}
Denoting by
\begin{eqnarray*}
\tilde L_h&=&\left(\|\nabla \bu\|_\infty+2\|\nabla \bw_h\|_\infty+\frac{1}{T}\right)+\left(C\frac{\delta_{\rm max}}{\bu_0}\|\nabla \bw_h\|_\infty^2 h\right)\\
&\le&\hat L_h:=\left(\|\nabla \bu\|_{L^\infty(0,T,L^\infty)}+2C_{1,\bw_h}+\frac{1}{T}\right)+\left(C\frac{\delta_{\rm max}}{\bu_0}C_{1,\bw_h}^2 h\right)
\end{eqnarray*}
we reach
\begin{eqnarray*}
&&\frac{d}{dt}\|\be_h\|_0^2+{\nu}\|\nabla \be_h\|_0^2+S(\be_h,\be_h)\le\hat L_h{\|\be_h\|_0^2}\nonumber\\
&&\quad+CTh^{2k+2}\|\bu_t\|_{k+1}+C_{\bu,\bw_h}^1\|\bz-\bpsi_h\|_*^2+C TC_{\bw_h}^2\|\nabla \bbeta_h\|_0^2,
\end{eqnarray*}
where
$$
C^1_{\bu,\bw_h}=C\left(\frac{h	}{v_0}\left\|\nabla\bu\right\|_{L^\infty(0,T,L^\infty)}+1+\frac{C_{\bu,\bw_h}}{v_0}\right).
$$
Applying Gronwall's Lemma we get
\begin{align}\label{prop1_k}
\left\|\be_h(t)\right\|_0^2 \le& \exp\left(\hat L_h t\right)\left(\left\|\be_h(0)\right\|_0^2 + CTh^{2k+2}\int_0^T \exp(-\hat L_hs) \left\|\bu_t\right\|_{k+1}^2\,dt\right.
\nonumber\\
& \left. \vphantom{\int_0^T }
{} + TC^1_{\bu,\bw_h} \max_{0\le t\le T} \left\| \bz -\bpsi_h\right\|_{*}^2 + CT^2C_{\bw_h}^2 \left\|\nabla \bbeta_h\right\|_{L^\infty(0,T,L^2)}^2
\right),
\end{align}
so that applying \eqref{cota_Ph_z} the proof is finished.
\end{proof}
\begin{remark} Notice that assumption~\eqref{u0} is not restrictive from the practical point of view. If $\left\|\bu_h\right\|_\infty$ approaches 0,
then the stabilization terms $S_1$, $S_2$ and~$S_3$ can be redefined so that the factor $\left\|\bu_h\right\|_{\infty}^{-1}$ is replaced by
$\max(u_0,\left\|\bu_h\right\|_{\infty})^{-1}$, for some $u_0>0$.
\end{remark}
\begin{remark} 
Let us observe that the only term of order $2k$ in \eqref{prop1_k} is the last one that comes from the
bound of $(\bw_h\cdot\nabla \bbeta_h,\be_h)$. To improve the rate of convergence we need to apply a different bound for this
term or, analogously, for the term $(\bu_h\cdot\nabla \bbeta_h,\be_h)$. This is carried out in next lemma. To be
able to improve the rate of convergence the added stabilization terms are essential in the error bound of this term. 
\end{remark}
\begin{lema}\label{cons1} (Consistency of the nonlinear term). 
\rm Let $C_{\bu_h}>0$ and~$\Gamma>0$ be such that
\begin{equation}
\label{cota_uh}
\left\| \bu_h\right\|_{\infty} \le C_{\bu_h}
\end{equation}
and
\begin{equation}
\label {threshold}
\left\| \be_h\right\|_0 \le \Gamma h^2,
\end{equation}
hold. Then, for any~$\varepsilon>0$, the following bound holds,
\begin{align}
\label{cota23}
\left| ((\bu_h\cdot\nabla)\bbeta_h,\be_h)\right| \le& 
 \frac{C_{\bu,h}^2}{v_0} \left\|\be_h\right\|_0^2 + \varepsilon S(\be_h,\be_h)
\nonumber\\
&\quad{} +\left(\left(1+C_{\ref{lema_41}}^2\right)\frac{C_{\bu_h}}{2\varepsilon\delta_{\min} v_0}+C_1\right)
\left\| \bz-\bpsi_h\right\|_{*}^2,
\end{align}
where $C_{2,1}$ is the constant in Lemma~\ref{lema_41},
\begin{equation}
\label{C2uh}
C_{\bu,h}=c_{\rm inv}h^{(3-d))/2} \left(\Gamma + Ch^{k-1}  \| \bz\|_{k+2}+ C\|\bu\|_3h\right),
\end{equation}
if $k=2,3$, $C_{\bu,h}=0$ and $C_1=0$. Otherwise, $C_1$ is a constant which depends only on the dimension~$d$, $c_{\rm inv}$, $C_{\ref{lema_41}}$ and~the constant $C_{\rm tr}$ in~\eqref{trace}.
\end{lema}
\begin{proof}
As in~\cite{barren_cip}, we write
\begin{equation}
\label{segundo_paso_a}
 \left| ((\bu_h\cdot\nabla)\be_h,\bbeta_h)\right|  \le 
\left|\langle \llbracket (\bu_h\cdot\nabla)\be_h \times {\bn}\rrbracket, \bz-\bpsi_h\rangle_{{\cal F}^i}\right| +\left|({\cal B}\be_h,\bz-\bpsi_h)_h\right|.
\end{equation}
For the first term on the right-hand side above, we have
\begin{align*}
\left|\langle \llbracket (\bu_h\cdot\nabla)\be_h \times {\bn}\rrbracket, \bz-\bpsi_h\rangle_{{\cal F}^i}\right|&{}\le
S_1(\be_h,\be_h)^{1/2}\frac{\|\bu_h\|_\infty^{1/2}}{h} \left\| \bz -\bpsi_h\right\|_{{\cal F}^i}
\nonumber\\
&{}\le \frac{\epsilon}{2} S(\be_h,\be_h) + \frac{C_{\bu_h}}{2\varepsilon\delta_{\min} h^2} \left\| \bz -\bpsi_h\right\|_{{\cal F}^i}^2
\nonumber\\
&{}\le  \frac{\epsilon}{2} S(\be_h,\be_h)  + \frac{C_{\bu_h}}{2\varepsilon\delta_{\min} v_0}\left\| \bz-\bpsi_h\right\|_{*}^2,
 \end{align*}
 where, in the last inequality, we have applied the trace inequality~\eqref{trace}.
 For the second term on the right-hand side of~\eqref{segundo_paso_a}, we now distinguish two cases.
 \medskip
 
 \noindent{\it Case $k=2,3$}
 
 We notice that, if $k=2$, ${\cal B}\be_h\in {\boldsymbol D}_h^2$ and, if $k=3$, ${\cal B}\be_h\in {\boldsymbol D}_h^4$ so that,
 for $k=2,3$ we have  ${\cal B}\be_h\in {\boldsymbol D}_h^{r+1}$ and we can argue as in~\cite{barren_cip}. Being~$\bpsi_h$ the orthogonal projection of~$\bz$ onto~${\boldsymbol Z}_h$ we have that~$\bz -\bpsi_h$ is orthogonal to any~$\bm_h\in{\boldsymbol Z}_h$, so that we can write
 \begin{align}
 \label{tercer_paso_10}
 \left|({\cal B}\be_h,\bz-\bpsi_h)_h\right| &=  \left|({\cal B}\be_h - \bm_h,\bz-\bpsi_h)_h \right|
 \\
 &{}\le \inf_{\bm\in{\boldsymbol Z}_h} \left\|h^{3/2} (({\cal B}\be_h - \bm_h)\right\|_h \left\|h^{-3/2} (\bz-\bpsi_h)\right\|_h
 \nonumber
 \\
 &{} \le C_{\ref{lema_41}}\left(\left\| h^2\llbracket {\cal B}\be_h\rrbracket\right\|_{{\cal F}^i} + \left\| h^3\llbracket \nabla {\cal B}\be_h\rrbracket\right\|_{{\cal F}^i}\right)
  \left\|h^{-3/2} (\bz-\bpsi_h)\right\|_h
  \nonumber
 \end{align}
 where, in the last inequality, we have applied Lemma \ref{lema_41}.  Thus, we have
   \begin{align}
 \label{tercer_paso_12}
 \left|({\cal B}\be_h,\bz-\bpsi_h)_h\right| &\le  \frac{\epsilon}{2} S(\be_h,\be_h) +C_{\ref{lema_41}}^2 \frac{C_{\bu_h}}{2\varepsilon\delta_{\min} h^3}  \left\| \bz-\bpsi_h\right\|_h^2
 \nonumber\\
 &{}\le   \frac{\epsilon}{2} S(\be_h,\be_h) + C_{\ref{lema_41}}^2\frac{C_{\bu_h}}{2\varepsilon\delta_{\min} v_0}   \left\| \bz-\bpsi_h\right\|_{*}^2
 \end{align}
 
\medskip
 \noindent{\it Case $k\ge 4$}. We argue as in~\cite{barren_cip} and we write
 \begin{equation}
 \label{decomp1}
 \left|({\cal B}\be_h,\bz-\bpsi_h)_h\right|\le  \left|(({\cal B}\be_h-{\cal B}_h)\be_h,\bz-\bpsi_h)_h\right| + \left|({\cal B}_h\be_h,\bz-\bpsi_h)_h\right|
\end{equation}
where
 $$
 {\cal B}_h\bw\mid_K\equiv \hbox{\rm curl}
 ((\bs_h
 \cdot\nabla) \bw_K),
 \qquad\forall K\in~\mathcal{T}_{h},
 $$
 where $\bs_h$ is the modified stokes projection onto quadratic finite elements.
 Observe then that ${\cal B}_hw_h \in D^{r+1}_h$. 
 Applying inverse inequality \eqref{inv} we have
\begin{align}
\label{u-sh-0}
h^{3/2}\left\| \bu_h-\bs_h\right\|_\infty &\le c_{\rm inv} h^{(3-d)/2}\| \bu_h-\bs_h\|_0
\nonumber\\
&\le c_{\rm inv} h^{(3-d)/2}\left( \|\bu_h-\bw_h\|_0+\| \bw_h-\bu \|_0+ \| \bu -\bs_h\|_0\right)
\nonumber\\
&\le c_{\rm inv}h^{(3-d)/2} \left(\Gamma + Ch^{r-1}  \| \bz\|_{r+2}+ C\|\bu\|_3h\right)h^2,
\nonumber\\
&= C_{\bu,h} h^2,
\end{align}
and
\begin{align}
\label{u-sh-1}
h^{3/2}\left\|\nabla(\bu_h-\bs_h)\right\|_\infty &\le c_{\rm inv} h^{(1-d)/2}\| \bu_h-\bs_h\|_0
\nonumber\\
&\le c_{\rm inv}h^{(3-d))/2} \left(\Gamma + Ch^{r-1}  \| \bz\|_{r+2}+ C\|\bu\|_3h\right)h\nonumber\\
&= C_{\bu,h} h.
\end{align}
For $K\in~\mathcal{T}_{h}$ we also have
\begin{align}
\label{u-sh-2}
h^{5/2}\left\|\hbox{\rm Hess}(\bu_h-\bs_h)\right\|_{\infty,K} &\le c_{\rm inv} h^{(1-d)/2}\| \bu_h-\bs_h\|_{0}
\nonumber\\
&\le c_{\rm inv}h^{(3-d))/2} \left(\Gamma + Ch^{r-1}  \| \bz\|_{r+1}+ C\|\bu\|_3h\right)h\nonumber\\
&= C_{\bu,h} h.
\end{align}
Consecuently, applying \eqref{u-sh-0}, \eqref{u-sh-1} and inverse inequality, for the first term on the right-hand side of~\eqref{decomp1} we obtain
\begin{align}
\label{paso_22}
\left| \left(({\cal B}-{\cal B}_h)\be_h,\bz-\bpsi_h\right)_h\right| &= \left| \left(h^{3/2}({\cal B}-{\cal B}_h)\be_h,h^{-3/2}(\bz-\bpsi_h)\right)_h\right|
\nonumber\\
&{}\le 
C_{\cal B}C_{\bu,h}c_{\rm inv}
\left\| \be_h\right\|_0\frac{\| \bz -\bpsi_h\|_{*}}{v_0^{1/2}}
\nonumber
\\
&{}\le  \frac{C_{\bu,h}^2}{2v_0} \left\| \be_h\right\|^2
+ \frac{c_{\rm inv}^2C_{\cal B}^2}{2} \| \bz -\bpsi_h\|_{*}^2.
\end{align}
(Observe that the constant~$C_{\cal B}$  above depends only on the number of terms in~${\cal B}$, which, in turn, depends on the dimension).
Also, for $({\cal B}_h\be_h,\bz-\bpsi_h)_h$ we write
\begin{align}
 \label{tercer_paso_20}
 \left|({\cal B}_h\be_h,\bz-\bpsi_h)_h\right| &=  \left|({\cal B}_h\be_h - \bm_h,\bz-\bpsi_h)_h \right|
 \\
 &{}\le \inf_{\bm\in{\boldsymbol Z}_h} \left\|h^{3/2} (({\cal B}_h\be_h - \bm_h)\right\|_h \left\|h^{-3/2} (\bz-\bpsi_h)\right\|_h
 \nonumber
 \\
 &{} \le C_{\ref{lema_41}}\left(\left\| h^2\llbracket {\cal B}_h\be_h\rrbracket\right\|_{{\cal F}^i} + \left\| h^3\llbracket \nabla {\cal B}_h\be_h\rrbracket\right\|_{{\cal F}^i}\right)
  \left\|h^{-3/2} (\bz-\bpsi_h)\right\|_h,
  \nonumber
 \end{align}
  where, as before, in the last inequality we have applied Lemma \ref{lema_41}. 
  Thus, from~\eqref{decomp1}, \eqref{paso_22} and \eqref{tercer_paso_20} and 
  and adding and subtracting ${\cal B}\be_h$, we may write
  \begin{align}
 \label{tercer_paso_22}
 &\left|({\cal B}\be_h,\bz-\bpsi_h)_h\right| \le\nonumber\\
 &\quad C_{\ref{lema_41}}\left(\left\| h^2\llbracket ({\cal B}_h-{\cal B})\be_h\rrbracket\right\|_{{\cal F}^i} + \left\| h^3\llbracket \nabla ({\cal B}_h-{\cal B})\be_h\rrbracket\right\|_{{\cal F}^i}\right) \left\|h^{-3/2} (\bz-\bpsi_h)\right\|_h
 \nonumber\\
& \quad{}+C_{\ref{lema_41}}\left(\left\| h^2\llbracket ({\cal B}\be_h\rrbracket\right\|_{{\cal F}^i} + \left\| h^3\llbracket \nabla {\cal B}\be_h\rrbracket\right\|_{{\cal F}^i}\right) \left\|h^{-3/2} (\bz-\bpsi_h)\right\|_h
 \nonumber\\
&\quad{} +  \frac{C_{\bu,h}^2}{2v_0} \left\| \be_h\right\|^2
+ \frac{c_{\rm inv}^2C_{\cal B}^2}{2}  \| \bz -\bpsi_h\|_{*}^2.
 \end{align}
In view of \eqref{u-sh-0} and~\eqref{u-sh-1} and applying inverse inequality \eqref{inv} we have
\begin{align}
\label{paso_231}
\left\| h^2\llbracket ({\cal B}_h-{\cal B})\be_h\rrbracket\right\|_{{\cal F}^i} &\le CC_{\bu,h}\left\| h^{1/2}\be_h\right\|_{{\cal F}^i}
\nonumber\\
&{}\le CC_{\bu,h}(\|\be_h\|_0 + h\|\nabla\be_h\|_0)
\nonumber\\
&  CC_{\bu,h}(1+c_{\rm inv})\|\be_h\|_0.
\end{align}
Arguing as before and applying  \eqref{u-sh-0}, \eqref{u-sh-1} and~\eqref{u-sh-2} together with inverse inequality \eqref{inv} we have
\begin{align}
\label{paso_232}
\left\| h^3\llbracket \nabla ({\cal B}_h-{\cal B})\be_h\rrbracket\right\|_{{\cal F}^i} &\le CC_{\bu,h}\left\| h^{1/2}\be_h\right\|_{{\cal F}^i}
\nonumber\\
&{}\le CC_{\bu,h}(\|\be_h\|_0 + h\|\nabla\be_h\|_0)
\nonumber\\
&  CC_{\bu,h}(1+c_{\rm inv})\|\be_h\|_0.
\end{align}
Thus, in view of \eqref{tercer_paso_22}, \eqref{paso_231} and~\eqref{paso_232}, we conclude
\begin{align}
 \label{tercer_paso_4}
 \left|({\cal B}\be_h,\bz-\bpsi_h)_h\right| &\le
 \frac{\varepsilon}{2} S(\be_h,\be_h) + C_{\ref{lema_41}}^2\frac{C_{\bu_h}}{2\varepsilon \delta_{\min}d v_0}  \| \bz -\bpsi_h\|_{*}^2
 \nonumber\\
 &\quad {}+   \frac{C_{\bu,h}^2}{v_0} \left\| \be_h\right\|^2
+ C \| \bz -\bpsi_h\|_{*}^2.
 \end{align}
 Observe that the constants~$C$ in~\eqref{paso_231} and~\eqref{paso_232}, depend on the number of terms of~${\cal B}$ (hence, on the dimension~$d$), the constant $C_{\rm tr}$ in the trace inequality~\eqref{trace}, and the constant $c_{\rm inv}$ in the inverse inequatlity \eqref{inv}. Since, as mentioned, the constant~$C_{\cal B}$ multiplying  $\| \bz -\bpsi_h\|_*^2$ in~\eqref{tercer_paso_22} depends on the number of terms in~${\cal B}$, we conclude that the constant~$C$ in~\eqref{tercer_paso_4} depends on the dimension, $C_{\rm tr}$ $c_{\rm inv}$ and the constant~$C_{\ref{lema_41}}$ in~Lemma~\ref{lema_41}.
\end{proof}
In next theorem we prove the main result of the paper.
\begin{Theorem}\label{th:main} Let $k\ge 2$ and assume that 
$\bu \in L^{\infty}(0,T,H^{k+1})$ and that $u_t\in L^2(0,T,H^{k+1})$.  Assume also \eqref{u0}
and that
\begin{equation}
\label{cond_e0}
\left\| \be_h(0)\right\|_0\le c_1h^{k+1/2},
\end{equation}
for some constant~$c_1>0$.
Then, there exist a positive constant $C_3$ such that the following bound holds,
\begin{eqnarray*}
&&\max_{0\le t\le T}\left\| \be_h(t)\right\|_0 +\left(\int_0^TS(\be_h(t),\be_h(t))dt\right)^{1/2}\le\\
&&\quad C_3\left(\left\|\be_h(0)\right\| + h^{k+1}\biggl(T\int_0^T \left\|\bu_t(t)\right\|_{k}^2\,dt\biggr)^{1/2}+ T^{1/2}\max_{0\le t\le T}
\left\|\bz(t)-\bpsi_h(t)\right\|_{*}\right).
\end{eqnarray*}
\end{Theorem}
\begin{proof}
We first notice that by writing $\bu_h = \bw_h + (\bu_h-\bw_h)$, from Proposition~\ref{Prop0}  and the inverse inequality it follows that
\begin{equation}
\label{apriori1}
\left\| \bu_h(t)\right\|_\infty \le\left\|\bw_h(t)\right\|_\infty+ c_{\rm inv}C_0h^{(2k-d)/2} 
\le C_{\bu_h}:=C_{\bw_h}+c_{\rm inv}C_0h^{(2k-d)/2}.
\end{equation}
Starting from \eqref{eq_error_3}
\begin{eqnarray}\label{lasti}
\frac{d}{dt}\|\be_h\|_0^2+{\nu}\|\nabla \be_h\|_0^2+S(\be_h,\be_h)\le\left(3\|\nabla \bu\|_\infty+\frac{1}{T}\right){\|\be_h\|_0^2}+2(\bu_h\cdot\nabla\bbeta_h,\be_h)\nonumber\\
\quad+S(\bw_h,\bw_h)+CTh^{2k+2}\|\bu_t\|_{k+1}+C\left(\frac{h	}{v_0}\left\|\nabla\bu\right\|_\infty+1\right)
\|\bz-\bpsi_h\|_*^2,
\end{eqnarray}
we apply Lemmas \ref{cons1} and \ref{cons2} to bound the two remaining terms.
To bound the stabilization term in view of \eqref{eq:Sww} we get
\begin{eqnarray*}
S(\bw_h,\bw_h)\le \left(C\frac{\delta_{\rm max}}{u_0}\|\nabla \bw_h\|_\infty^2 h\right)\|\be_h\|_0^2
+\frac{C_{\bu,\bw_h}}{v_0}\left\| \bz-\bpsi_h\right\|_{*}^2,
\end{eqnarray*}
where $C_{\bu,\bw_h}$ is the constant in \eqref{Cuw}.
For the nonlinear term we apply Lemma~\ref{cons1} with $\varepsilon=1/4$. Then
\begin{align*}
2\left| ((\bu_h\cdot\nabla)\bbeta_h,\be_h)\right| \le& 
 \frac{2C_{\bu,h}^2}{v_0} \left\|\be_h\right\|_0^2 + \frac{1}{2} S(\be_h,\be_h)
\nonumber\\
&\quad{} +\left(\left(1+C_{\ref{lema_41}}^2\right)\frac{4C_{\bu_h}}{\delta_{\min} v_0}+2C_1\right)
\left\| \bz-\bpsi_h\right\|_{*}^2.
\end{align*}
Inserting the two bounds above into \eqref{lasti} we get
\begin{eqnarray*}
\frac{d}{dt}\|\be_h\|_0^2+{\nu}\|\nabla \be_h\|_0^2+\frac{1}{2}S(\be_h,\be_h)&\le&\hat L_h \|\be_h\|_0^2+CTh^{2k+2}\|\bu_t\|_{k+1}
\\
&&\quad+g_h\left\| \bz-\bpsi_h\right\|_{*}^2,
\end{eqnarray*}
where
\begin{eqnarray*}
\hat L_h&=&3\|\nabla \bu\|_{L^\infty(0,T,L^\infty)}+\frac{1}{T}+C\frac{\delta_{\rm max}}{\bu_0}C_{1,\bw_h}^2 h+\frac{2}{v_0}\max_{0 \le t\le T}C_{\bu,h}^2,\\
g_h&=&C\left(\frac{h	}{v_0}\left\|\nabla\bu\right\|_{L^\infty(0,T,L^\infty)}+1\right)+\frac{C_{\bu,\bw_h}}{v_0}
+\left(\left(1+C_{\ref{lema_41}}^2\right)\frac{4C_{\bu_h}}{\delta_{\min} v_0}+2C_1\right).
\end{eqnarray*}
Thus, applying Gronwall's Lemma as in the proof of Proposition~\ref{Prop0} the proof is finished.
\end{proof}
\begin{corollary}
In the conditions of Theorem~\ref{th:main} there exist a constant~$C_4$ such that the following bound holds for $h\ge \mu/v_0$ and~$t\in [0,T]$:
$$
\left\|\be_h(t)\right\|_0\le C_4h^{k+1/2}.
$$
Applying triangle inequality together with \eqref{cota_Ph_z} we conclude
$$
\left\|\bu(t)-\bu_h(t)\right\|_0\le C_4h^{k+1/2}.
$$
\end{corollary}
\begin{remark}
Since we are using inf-sup stable elements the error analysis of the pressure is standard. One can follow the steps of
\cite[Theorem 2]{nos_grad_div} applying \cite[Lemma 1]{cor} to bound $\|\be_{h,t}\|_{-1}$ and taking into account that for
$\bv_h\in \bV_h$ it holds $S(\bv_h,\bv_h)\le C h^{1/2}\|\bv_h\|_1$. Then, assuming conditions of Theorem \ref{th:main} hold it is not difficult to prove the following bound
for the pressure
$$
\left\|p(t)-p_h(t)\right\|_0\le Ch^{k}.
$$
\end{remark}
\section{Numerical Experiments}
We check the order of convergence shown in Theorem~\ref{th:main}. For this purpose,
we consider the Navier-Stokes equations in  the domain $\Omega=[0,1]^2$ with $T=4$, and the forcing term~$\bff$ chosen so that the solution $\bu$ and~$p$ are given by
\begin{eqnarray}
\label{eq:exactau}
\bu(x,y,t)&=& \frac{6+4\cos(4t)}{10} \left[\begin{array}{c} 8\sin^2(\pi x) 2y(1-y)(1-2y)\\
-8\pi\sin(2\pi x) (y(1-y))^2\end{array}\right],\\
p(x,y,t)&=&\frac{6+4\cos(4t)}{10} \sin(\pi x)\cos(\pi y).
\label{eq:exactap}
\end{eqnarray}
The initial condition was set to the Stokes projection  of~$u(0)$ onto~$V_h$.

We present results for the case $k=2$ on barycentric-refined grids of regular triangulations with SW-NE diagonals,
with $N=6$, $12$, $24$ and~$48$ subdivisions in each coordinate direction. In Fig.~\ref{red} we can see the coarsest mesh used. For the CIP stabilized method the parameters were set to $\gamma_1=10^{-2}$, $\gamma_2=10^{-4}$ and~$\gamma_3=10^{-3}$. These values were found to produce the best results after a simple and quick search starting from the values used in~\cite{barren_cip}.
\begin{figure}
\begin{center}
\includegraphics[height=5truecm]{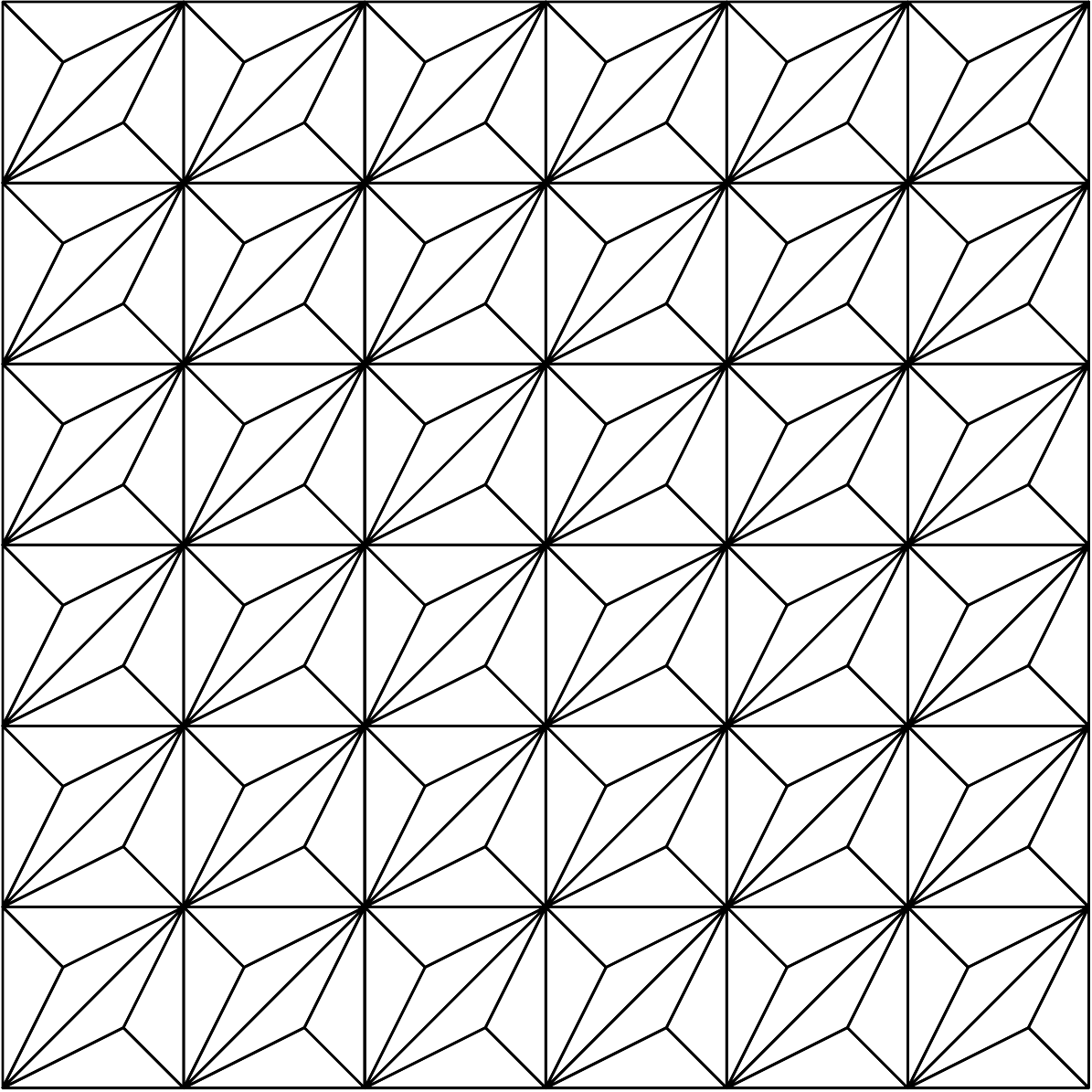}
\caption{Coarsest mesh used\label{red}}
\end{center}
\end{figure}

The time integrator used was a semi-implicit variable step size variable order code described
in~\cite{BDF2}. The tolerances for the local errors were chosen sufficiently small so that the dominant errors were those arising from the spatial discretization. We obtained approximations every $0.1$ units of time until $T=4$ and computed the maximum error in $L^2$ of these approximation.
In Fig.~\ref{svsvcip} (left plot) we show the errors of the CIP-stabilized Scott-Vogelius approximation for $\nu=10^{-8}$. We also show the results of the same method without stabilization. It can be seen a dramatic improvement when stabilization is added. Also we show the slopes of a least squares fit to the three smallest errors of each method, which indicate the decay rate with the mesh diameter~$h$. We can see that rate of the stabilized method is 1/2 larger than that of the unstabilized method, as predicted by the theory, and close to the theoretical value of~2.5.
\begin{figure}[h]
\begin{center}
\includegraphics[height=5truecm]{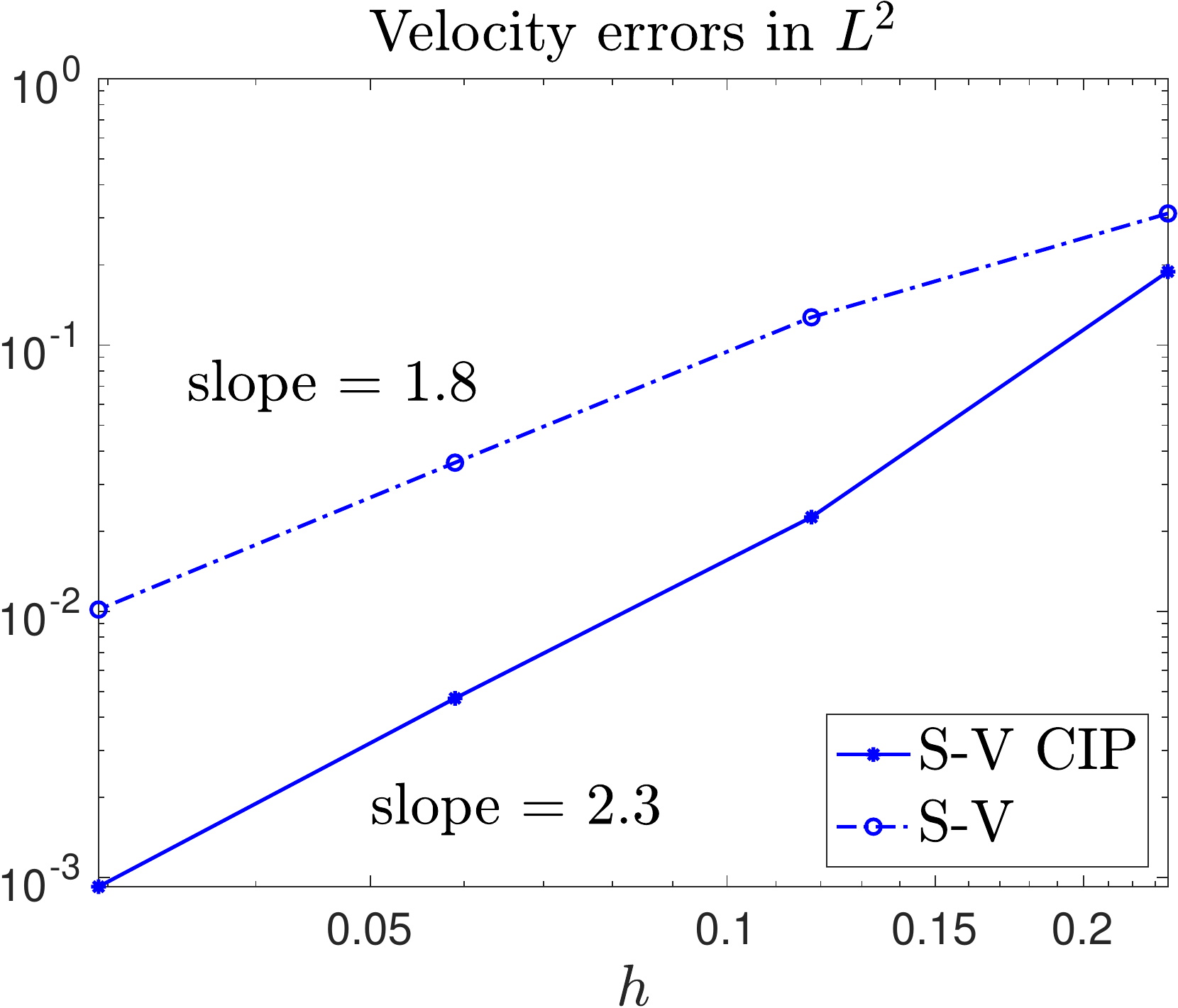}
\includegraphics[height=5truecm]{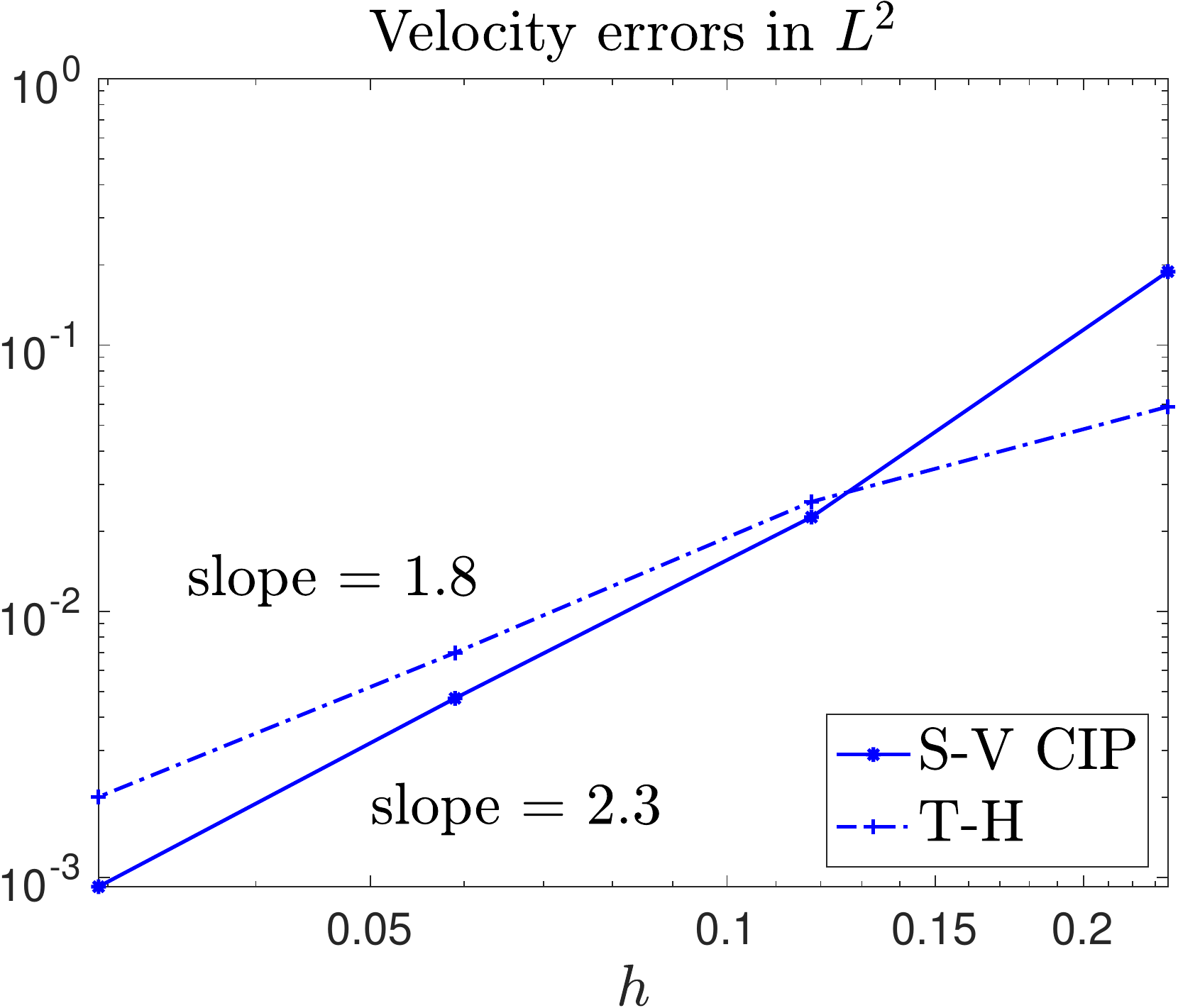} 
\caption{Velocity errors in $L^2$. Left: Scott-Vogelius elements with and without CIP-stabilization. Right: CIP-stabilized Scott-Vogelius elements and Taylor-Hood elements with grad-div stabilization.\label{svsvcip}}
\end{center}
\end{figure}
On the right-plot we also show the errors of Taylor-Hood elements with grad-div stabilization (the parameter for the grad-div term was set to~$0.05$, as in~\cite{grad-div} and~\cite{BDF2}). It is well-known that, due to the larger pressure constraint,  Scott-Vogelius elements are usually less accurate that stabilized Taylor-Hood elements, unless in the presence of large pressure gradients. This is indeed the case in the present example, where we can see that the errors of the unstabilized Scott-Vogelius on the left plot are five times larger than those of the stabilized Taylor-Hood elements on the right-plot. However, thanks to its better convergence rate, the CIP-stabilized Scott-Vogelius method becomes more accurate than the stabilized Taylor-Hood method. For this, the rate of convergence in the present example is 1.8, close to the theoretical value of~2 (see e.g.~\cite{review}).

\bibliographystyle{abbrv}
\bibliography{references}
\end{document}